\documentclass{amsart}

\usepackage{comment}
\usepackage{amsmath}
\usepackage{amsthm}
\usepackage{amssymb}
\usepackage{amsfonts}
\usepackage[shortlabels]{enumitem}
\usepackage[bookmarks=true,hyperindex,pdftex,colorlinks,citecolor=red,linkcolor=cyan]{hyperref}
\usepackage[all]{xy}

\usepackage{marginnote} 
\usepackage{xcolor}
\newcommand{\clch}[1]{{\color{violet}{#1}}} 

\DeclareMathOperator{\lspan}{span}
\DeclareMathOperator{\supp}{supp}
\DeclareMathOperator{\Lip}{Lip}

\newcommand{\N}{\mathbb{N}}
\newcommand{\Q}{\mathbb{Q}}
\newcommand{\R}{\mathbb{R}}
\newcommand{\C}{\mathbb{C}}
\newcommand{\K}{\mathbb{K}}

\newcommand{\F}{\mathcal{F}}
\newcommand{\Orb}{\mathrm{Orb}}

\theoremstyle{plain}
\newtheorem{theorem}{Theorem}[section]
\newtheorem{lemma}[theorem]{Lemma}
\newtheorem{corollary}[theorem]{Corollary}
\newtheorem{proposition}[theorem]{Proposition}

\theoremstyle{definition}
\newtheorem{definition}[theorem]{Definition}
\newtheorem{example}[theorem]{Example}
\newtheorem{question}{Question}

\theoremstyle{remark}
\newtheorem{remark}[theorem]{Remark}

\begin{document}
	
	\title[Invariant subspaces for free linearizations of Lipschitz maps]{Invariant subspaces for free linearizations of Lipschitz maps}
	\date{} 
	

	\author[C. Coine]{Cl\'ement Coine}
	\address[C. Coine]{UNICAEN, CNRS, LMNO, 14000 Caen, France}
	\email{clement.coine@unicaen.fr}
	
	\author[P. L. Kaufmann]{Pedro L. Kaufmann}
	\address[P. L. Kaufmann]{Universidade Federal de S\~ao Paulo - UNIFESP. Instituto de Ci\^encia e Tecnologia. Departamento de Matem\'atica. S\~ao Jos\'e dos Campos - SP, Brasil}
	\email{\texttt{plkaufmann@unifesp.br}}
	
	\author[C. Petitjean]{Colin Petitjean}
	\address[C. Petitjean]{Univ Gustave Eiffel, Univ Paris Est Creteil, CNRS, LAMA UMR8050, F-77447 Marne-la-Vallée, France}
	\email{colin.petitjean@univ-eiffel.fr}
	
	\author[S. Tapia-Garc\'ia]{Sebasti\'an Tapia-Garc\'ia}
	\address[S. Tapia-Garc\'ia]{Institute of Statistics and Mathematical Methods in Economics, E105-04, TU Wien, Wiedner Hauptstraße 8, A-1040 Wien}
	\email{sebastian.tapia.garcia@tuwien.ac.at}

	\begin{abstract}
		We consider the invariant subspace problem for operators induced by Lipschitz self-maps on Lipschitz-free spaces.
		Besides the usual free linearizations $\widehat f$ of basepoint-preserving Lipschitz self-maps, we consider a wider class of operators $T_{f,e}$ which are naturally defined for arbitrary Lipschitz self-maps.
		We show, among other results, that every $\widehat f$ admits a non-trivial invariant subspace whenever the underlying metric space contains a compact ball, or has at least two connected components one of which has non-empty interior. We also obtain corresponding positive results for $T_{f,e}$ under isolated-point, compact-ball and disconnectedness assumptions, and discuss some consequences for linear dynamics.
	\end{abstract}
	
	
	\maketitle

	\section{Introduction}
	
	Let $M$ and $N$ be metric spaces. We denote by $\Lip(M,N)$ the set of all Lipschitz maps $f\colon M\to N$. If $M$ and $N$ are pointed metric spaces, with distinguished points $0_M\in M$ and $0_N\in N$, then $\Lip_0(M,N)$ denotes the set of all $f\in \Lip(M,N)$ such that $f(0_M)=0_N$. In either case, $\Lip(f)$ denotes the optimal Lipschitz constant of $f$, that is,
	$$
	\Lip(f):=\sup_{x\neq y\in M}\frac{d(f(x),f(y))}{d(x,y)}.
	$$
	Throughout the paper, $\K$ denotes either $\R$ or $\C$. If $M$ is pointed, we write
	$$
	\Lip_0(M):=\Lip_0(M,\K).
	$$
	Endowed with the norm $\|f\|_L:=\Lip(f)$, this is a Banach space. For $x\in M$, let $\delta_M(x)\in \Lip_0(M)^*$ be the evaluation functional defined by
	$$
	\langle \delta_M(x),f\rangle=f(x),\qquad f\in \Lip_0(M).
	$$
	The map
	$
	\delta_M\colon x\in M\longmapsto \delta_M(x)\in \Lip_0(M)^*
	$
	is an isometric embedding. The \emph{Lipschitz-free space over $M$} is defined by
	$$
	\F(M):=\overline{\mathrm{span}}^{\|\cdot\|}
	\{\delta_M(x):x\in M\}\subset \Lip_0(M)^*.
	$$
	It is well known that $\F(M)$ is a canonical predual of $\Lip_0(M)$; we shall write
	$$
	\Lip_0(M)\equiv \F(M)^*.
	$$
	Since the choice of the distinguished point will play a role in this paper, we shall write $\F(M,0)$ whenever we wish to emphasize that $0\in M$ is the chosen basepoint.
	\smallskip
	
	One of the most useful features of Lipschitz-free spaces is their linearization property. Let $M$ and $N$ be pointed metric spaces, and let $f\colon M\to N$ be a basepoint-preserving Lipschitz map, that is, $f(0_M)=0_N$. Then there exists a unique bounded linear operator
	$
	\widehat f\colon \F(M,0_M)\to \F(N,0_N)
	$
	such that $\|\widehat f\|=\Lip(f)$ and
	$$
	\widehat f(\delta_M(x))=\delta_N(f(x)),\qquad \forall  x\in M.
	$$
	Equivalently, the following diagram commutes:
	$$
	\xymatrix{
		M \ar[r]^f \ar[d]_{\delta_M}  & N \ar[d]^{\delta_N} \\
		\F(M,0_M) \ar[r]_{\widehat{f}} & \F(N,0_N).
	}
	$$
	In the literature, operators of the form $\widehat f\colon \F(M)\to \F(N)$ are sometimes called ``Lipschitz operators''. We shall instead use the terminology \emph{Lipschitz-free linearization}, or simply \emph{free linearization}, which is more precise and avoids possible ambiguity.
	
	\medskip
	
	A recent line of research aims to characterize linear properties of $\widehat f\colon \F(M)\to \F(N)$ in terms of metric properties of the original Lipschitz map $f\colon M\to N$. Some elementary instances of this principle are the following:
	\begin{enumerate}[$(i)$]
		\item $f$ has dense range if and only if $\widehat f$ has dense range;
		\item $f$ is bi-Lipschitz if and only if $\widehat f$ is an isomorphic embedding;
		\item $f$ is a surjective bi-Lipschitz map if and only if $\widehat f$ is an isomorphism;
		\item if $0\in N\subset M$ and $f\colon M\to N$ is a Lipschitz retraction, then $\widehat f$ is a projection from $\F(M)$ onto $\F(N)$.
	\end{enumerate}
	These facts are straightforward and can be found, for instance, in \cite{Weaver2}. Beyond these basic observations, however, the problem becomes substantially more delicate, even for such fundamental properties as injectivity or surjectivity; see, for example, \cite{Injectivity}. Other successful results in this direction include \cite{ACP21}, where compactness and weak compactness of $\widehat f$ are characterized, and \cite{curveflat}, where Dunford--Pettis and Radon--Nikodým free linearizations are studied.
	
	The present paper belongs to this general program. More precisely, we investigate the invariant subspace problem for free linearizations of Lipschitz self-maps.
	
	\begin{question}\label{Q1}
		Does every operator of the form
		$
		\widehat f\colon \F(M)\to \F(M)
		$
		admit a non-trivial invariant subspace? In other words, does there always exist a closed subspace $S\subset \F(M)$ such that $\{0\}\subsetneq S\subsetneq \F(M)$ and $\widehat f(S)\subset S$?
	\end{question}
	
	Here it is important that the domain and range spaces are exactly the same. This means not only that $f$ is a self-map of $M$, but also that the basepoints involved in the construction coincide. This is a significant restriction: in order to define $\widehat f\colon \F(M,0)\to \F(M,0)$, the map $f\colon M\to M$ must satisfy $f(0)=0$. To avoid this constraint, and thus to consider arbitrary Lipschitz self-maps, we shall use a closely related construction described below.
	\medskip
	
	Recall that changing the basepoint of $M$ does not affect the isometric structure of the corresponding Lipschitz-free space. More precisely, if $a,b\in M$ and $\delta_a : (M,a) \to \F(M,a)$ and $\delta_b : (M,b) \to \F(M,b)$ are the isometric embeddings, then the linear map
	$$
	\Phi_{a,b}:\delta_a(x) \in \F(M,a) \mapsto \delta_b(x)-\delta_b(a) \in \F(M,b)
	$$ 
	defines an onto linear isometry. Let now $M$ be a metric space, let $e\in M$, and let $f\colon M\to M$ be a Lipschitz map. By the linearization property, there is a unique bounded operator
	$
	\widehat f_e\colon \F(M,e)\to \F(M,f(e))
	$
	such that $\widehat f_e(\delta_e(x))=\delta_{f(e)}(f(x))$ for every $x \in M$.
	We then define
	$$
	T_{f,e}\colon \F(M,e)\longrightarrow \F(M,e),
	\qquad
	T_{f,e}:=\Phi_{f(e),e}\circ \widehat f_e.
	$$
	Thus
	$$
	T_{f,e}(\delta_e(x))=\delta_e(f(x))-\delta_e(f(e)),
	\qquad \forall  x\in M.
	$$
	This is illustrated by the following diagram:
	\begin{equation*}
		\xymatrix{
			M \ar[r]^f  \ar[d]_{\delta_{e}} & M \ar[d]^{\delta_{f(e)}}\\
			\F(M,e) \ar[r]^{\widehat{f}_e}\ar[dr]^{T_{f,e}} & \F(M,f(e))\ar[d]^{\Phi_{f(e),e}} \\
			& \F(M,e)
		}
	\end{equation*}
	It is obvious but important to note that $f(e)=e$ if and only if $T_{f,e} = \widehat{f_e}$.
	
	\begin{remark}
		Equivalently, $T_{f,e}$ can be defined as the unique bounded linear operator on $\F(M,e)$ such that the following diagram commutes:
		\begin{equation*}
			\xymatrix{
				M \ar[r]^f  \ar[d]^{\delta_e} & M \ar[d]^{\delta_e-\delta_e(f(e))}\\
				\F(M,e) \ar@{->}[r]^{T_{f,e}} & \F(M,e)
			}
		\end{equation*}
		In particular, $\|T_{f,e}\|=\Lip(f)$. Moreover, for every $x\in M$ and every $n\in\N$, one has
		$$
		T_{f,e}^n(\delta_e(x))
		=
		\delta_e(f^n(x))-\delta_e(f^n(e)).
		$$
		This formula will be used repeatedly in the sequel.
	\end{remark}
	
	As explained above, the basepoint is constrained for operators of the form $\widehat f$. By contrast, the following proposition shows that, for the operators $T_{f,e}$, the particular choice of basepoint is immaterial up to conjugacy. Recall that two bounded linear operators $T:X\to X$ and $S:Y\to Y$ are said to be \emph{(isometrically) conjugate} if there exists an (isometric) isomorphism $\Phi:X\to Y$ such that $S\circ \Phi=\Phi\circ T$.
	
	\begin{proposition}\label{prop:conjugate}
		Let $f\colon M\to M$ be a Lipschitz map, and let $0,e\in M$. Then
		$$
		T_{f,0}\colon \F(M,0)\to \F(M,0)
		\qquad\text{and}\qquad
		T_{f,e}\colon \F(M,e)\to \F(M,e)
		$$
		are isometrically conjugate.
	\end{proposition}
	
	\begin{proof}
		It is enough to check the identity on elements of the form $\delta_0(x)$. For $x\in M$, we have
		$$
		\begin{aligned}
			\Phi_{0,e}T_{f,0}(\delta_0(x))
			&= \Phi_{0,e}\bigl(\delta_0(f(x))-\delta_0(f(0))\bigr) \\
			&= \delta_e(f(x))-\delta_e(f(0)).
		\end{aligned}
		$$
		On the other hand,
		$$
		\begin{aligned}
			T_{f,e}\Phi_{0,e}(\delta_0(x))
			&= T_{f,e}\bigl(\delta_e(x)-\delta_e(0)\bigr) \\
			&= \bigl(\delta_e(f(x))-\delta_e(f(e))\bigr)
			-\bigl(\delta_e(f(0))-\delta_e(f(e))\bigr) \\
			&= \delta_e(f(x))-\delta_e(f(0)).
		\end{aligned}
		$$
		By linearity and density, we deduce 
		$$
		\Phi_{0,e}\circ T_{f,0}=T_{f,e}\circ \Phi_{0,e}.
		$$
		Since $\Phi_{0,e}$ is an isometric isomorphism, the two operators are isometrically conjugate.
	\end{proof}
	
	\begin{remark} Since invariant subspaces are preserved under conjugacy, Proposition~\ref{prop:conjugate} implies that, in what concerns the presence of non-trivial invariant subspaces for operators of the form $T_{f,e}$, the choice of basepoint is irrelevant. More precisely, given a Lipschitz self-map $f:M\to M$ and $e,e'\in M$, $T_{f,e}$ has a non-trivial invariant subspace if and only if $T_{f,e'}$ does.
	\end{remark}
	We are in position to pose the following companion to Question 1:
	
	\begin{question}\label{Q2}
		Does every operator of the form $T_{f,e}\colon \F(M,e)\to \F(M,e)$
		admit a non-trivial invariant subspace?
	\end{question}
	
	\medskip

	The invariant subspace problem for reflexive Banach spaces remains one of the central open problems in operator theory: it is still unknown whether every bounded operator on a separable infinite-dimensional reflexive Banach space admits a non-trivial invariant closed subspace. In the wider class of separable Banach spaces, counterexamples are known. The first one was constructed by Enflo \cite{Enflo} and later simplified by Beauzamy \cite{Beauzamy}. Read subsequently produced further examples, including examples on classical spaces such as $\ell_1$ and $c_0$; see \cite{Read1,Read2}. These constructions are famously intricate.
	
	One motivation for the present work is therefore the following question. Could some of the combinatorial complexity underlying these counterexamples be encoded in the seemingly simpler dynamics of a Lipschitz self-map of a suitably chosen metric space? At this stage, this is more a guiding intuition than a concrete strategy. Nevertheless, Lipschitz-free spaces have already proved to be a fruitful source of unexpected examples in Banach space theory. For instance, they appear in Kalton's construction of uniformly homeomorphic but non-isomorphic separable Banach spaces \cite{Kalton04}, in the examples of Albiac and Kalton of Lipschitz-isomorphic separable quasi-Banach spaces which are not linearly isomorphic \cite{AlbiacKalton}, and, more recently, in Veeorg's construction of an RNP space with a Daugavet point \cite{Veeorg}. These examples suggest that Lipschitz-free spaces provide a natural setting in which to test and refine classical problems in Banach space theory.
	
	In this paper, we do not answer Question~\ref{Q1} or Question~\ref{Q2} in full generality. Instead, we identify several situations in which operators of the form $\widehat f$ or $T_{f,e}$ necessarily admit non-trivial invariant subspaces.
	Notably, for free linearizations, we prove in Proposition~\ref{Prop-Connected-Comp-fhat} and Theorem~\ref{Thm:CompactBall} that every operator
	$
	\widehat f\colon \F(M)\to \F(M)
	$
	has a non-trivial invariant subspace whenever either
	\begin{itemize}
		\item $M$ has at least two connected components, one of which has non-empty interior; or
		\item $M$ contains a compact ball.
	\end{itemize}
	
	The situation for the operators $T_{f,e}$ is more delicate, and our results are correspondingly less sharp. For instance, Proposition~\ref{Prop: ISP2} and Theorem~\ref{Thm-Tfe-Compactball} show that every operator
	$
	T_{f,e}\colon \F(M,e)\to \F(M,e)
	$
	has a non-trivial invariant subspace whenever either
	\begin{itemize}
		\item $M$ has infinitely many connected components, one of which has non-empty interior; or
		\item $M$ is non-compact but contains a compact ball.
	\end{itemize}
	
	Thus the purpose of the paper is not to settle the invariant subspace problem for free linearizations, but rather to separate the accessible cases from the genuinely difficult ones. In this sense, the results below may be viewed as a first step, and as an invitation to further investigate the invariant subspace problem within the framework of Lipschitz-free spaces. 
	
	In the last section we turn to study dynamical properties of the operators $T_{f,e}$. In particular, we refer to cyclic, hypercyclic, (weakly) mixing and wild dynamics. Our results highlight both the similarities and differences of the behavior of the orbits of $\delta(x)$ under the action of the operators $\widehat{f}$ (i.e. when $f$ admits a fixed point) and $T_{f,e}$. Hence, this section can be regarded as a complement of \cite{ACP20,Tapia}. Our main findings are: 
	\begin{itemize}
		\item For $f\in \mathrm{Lip}(M,M)$, if $\mathrm{Orb}(x,f)$ is dense in $M$, then there is $e\in M$ such that $\delta(x)$ is cyclic vector for $T_{f,e}$.
		\item If $f$ is weakly mixing (mixing; weakly mixing and chaotic resp.) then $T_{f,e}$ is weakly mixing (mixing; weakly mixing and chaotic resp.).
		\item There is no wild operator of the form $T_{f,e}$.
	\end{itemize}

	\section{Invariant subspaces: basic definitions and reductions}
	
	Let $T\colon X\to X$ be a bounded operator on a Banach space $X$, and let $f\colon M\to M$ be a Lipschitz self-map of a metric space $M$. We shall use the following terminology.
	
	A \emph{non-trivial invariant subspace} (NTIS in short) for $T$ is a closed linear subspace $S\subset X$ such that
	$$
	\{0\}\subsetneq S\subsetneq X
	\qquad\text{and}\qquad
	T(S)\subset S.
	$$
	For $x\in M$, the orbit of $x$ under $f$ is
	$$
	\Orb(x,f):=\{f^n(x):n\in\N\cup\{0\}\}.
	$$
	For $x\in X$, we say that $x$ is a \emph{cyclic vector} for $T$ if
	$$
	\overline{\mathrm{span}}\,\Orb(x,T)=X,
	$$
	and we say that $x$ is a \emph{hypercyclic vector} for $T$ if
	$$
	\overline{\Orb(x,T)}=X.
	$$
	
	The finite-dimensional case will not be relevant for us. Indeed, over the complex field, every operator on a finite-dimensional Banach space of dimension at least two admits a non-trivial invariant subspace; over the real field, the same conclusion holds in dimension at least three. Moreover, if $X$ is non-separable, then every non-zero orbit generates a separable closed invariant subspace; hence every bounded operator on $X$ admits a non-trivial invariant subspace. Thus, in the sequel, we shall restrict our attention to infinite separable metric spaces $M$. In this case, $\F(M)$ is infinite-dimensional and separable as well.
	
	We shall also assume, without loss of generality, that $M$ is complete. Indeed, passing from $M$ to its completion does not change the associated Lipschitz-free space up to canonical isometric isomorphism.
	
	Next, it is well known, and immediate from the definitions, that $T$ has no non-trivial invariant subspace if and only if every non-zero vector of $X$ is cyclic for $T$. There is a formally analogous statement for closed invariant subsets: with the usual convention that every non-empty proper closed invariant subset is non-trivial, $f$ has no non-trivial invariant subset if and only if every orbit is dense in $M$. For our purposes, however, singleton invariant sets should be regarded as trivial. Indeed, if $f(a)=a$, then $T_{f,e}$ is conjugate to $T_{f,a}=\widehat f_a$, and the singleton $\{a\}$ only generates the zero subspace of $\F(M,a)$. Since we are interested in producing non-zero proper invariant subspaces on the operator side, we shall therefore use the following convention.
	
	\begin{definition}\label{def:Invariant-subset}
		Let $M$ be a metric space and let $f\in \Lip(M,M)$. A \emph{non-trivial closed invariant subset} for $f$ is a closed subset $A\subset M$ such that $A$ contains at least two points,
		$$
		A\subsetneq M,
		\qquad\text{and}\qquad
		f(A)\subset A.
		$$
	\end{definition}
	
	The advantage of this definition is that it does not involve any choice of basepoint. This is particularly convenient for the operators $T_{f,e}$, where the basepoint $e$ is fixed independently of the dynamics of $f$. The next proposition shows that every non-trivial closed invariant subset gives rise to a non-trivial invariant subspace of $\F(M,e)$. Before stating it, let us establish some notation simplifications to be adopted throughout the rest of this work. Whenever the choice of basepoint is clear from the context, we shall write \(\delta\) instead of \(\delta_0\) or \(\delta_e\). For clarity, we shall use the notation \(0\) for the basepoint when \(f(0)=0\) and the free linearization \(\widehat f\) is under consideration. By contrast, we shall use the notation \(e\) for the basepoint when \(f\) is not assumed to fix any distinguished point and the operator \(T_{f,e}\) is considered. 
	
	\begin{proposition}\label{Prop: ISP2}
		Let $f\in \Lip(M,M)$, let $e\in M$, and let $A\subset M$ be a non-trivial closed invariant subset for $f$. Then
		$$
		S_A:=
		\overline{\mathrm{span}}\{\delta(x)-\delta(y):x,y\in A\}
		\subset \F(M,e)
		$$
		is a non-trivial invariant subspace for $T_{f,e}$.
	\end{proposition}
	
	\begin{proof}
		For every $x,y\in A$, one has
		$$
		T_{f,e}(\delta(x)-\delta(y))
		=
		\delta(f(x))-\delta(f(y)).
		$$
		Since $f(A)\subset A$, it follows that
		$$
		T_{f,e}\bigl(\mathrm{span}\{\delta(x)-\delta(y):x,y\in A\}\bigr)
		\subset
		\mathrm{span}\{\delta(x)-\delta(y):x,y\in A\}.
		$$
		By density and continuity, we obtain
		$$
		T_{f,e}(S_A)\subset S_A.
		$$
		
		It remains to check that $S_A$ is non-trivial. Since $A$ contains at least two points, we may choose $a,b\in A$ with $a\neq b$. Then $\delta(a)-\delta(b)\neq 0$, and therefore $S_A\neq \{0\}$.
		
		Finally, since $A$ is a proper closed subset of $M$, we may choose $x\in M\setminus A$. Fix $a\in A$ and define
		$$
		\varphi(z):=d(z,A)-d(e,A),\qquad \forall z\in M.
		$$
		Then $\varphi\in \Lip_0(M,e)$. Moreover, $\varphi$ is constant on $A$, and hence
		$$
		S_A\subset \ker \varphi.
		$$
		On the other hand,
		$$
		\langle \delta(x)-\delta(a),\varphi\rangle
		=
		d(x,A)>0.
		$$
		Thus $\ker\varphi$ is a proper closed subspace of $\F(M,e)$ containing $S_A$, and consequently $S_A\neq \F(M,e)$.
	\end{proof}
	
	Let us note that if $e\in A$, then $S_A$ takes the simpler form
	$$
	S_A=
	\overline{\mathrm{span}}\{\delta(x):x\in A\} =:\F_M(A).
	$$
	In this case, $S_A$ is naturally identified with the Lipschitz-free space over $A$ with basepoint $e$.
	
	It is natural to ask whether the converse of Proposition~\ref{Prop: ISP2} holds. It does not: There are Lipschitz self-maps $f\in \Lip_0(M,M)$ which have no non-trivial closed invariant subset, while their free linearization $\widehat f\colon \F(M)\to \F(M)$ does have a non-trivial invariant subspace. We now present two examples showing that the converse can fail in a rather strong way.
	
	\begin{example}\label{example:TorusWith0}
		
		Let $M=\mathbb T :=\{e^{2\pi i t}:t\in\R\}$,
		endowed with the metric inherited from the Euclidean distance on $\C$. We take $1$ as the distinguished point of $M$. Let $\alpha\in \R\setminus\Q$, and define
		$$
		f(z)=e^{2\pi i\alpha}z,\qquad \forall z\in \mathbb T.
		$$
		Then $f\in \Lip(M,M)$. Since the irrational rotation of the circle is minimal, $f$ has no non-trivial closed invariant subset; see, for instance, \cite[Exercise~1.2.9]{GEPM}. Let
		\begin{align*}
			E
			& := \overline{\mathrm{span}}\{T_{f,1}^n (\delta(-1)) :n\in \N \cup \{0\} \}\\
			& = \overline{\mathrm{span}}\{ \delta(-e^{2\pi i \alpha n}) - \delta(e^{2\pi i \alpha n})     :n\in \N \cup \{0\} \} \\
			&= \overline{\mathrm{span}}\{ \delta(-e^{2\pi i t}) - \delta(e^{2\pi i t})     :t\in \mathbb R \}
		\end{align*}
		be the closed linear span of the orbit of $\delta(-1)$ under $T_{f,1}$. Then $E\neq \{0\}$ and noticing that $E \subset \ker(\varphi)$
		where $\varphi$ is the non-zero element of $\Lip_0(M)$ defined by $\varphi(e^{2\pi i t}) = \sin(4\pi t)$, it is clear that $E$ is a non-trivial invariant subspace for $T_{f,1}$.

	\end{example}
	
	
	
	Another illustration of this phenomenon comes from linear operators, using for instance Read's examples of operators with no non-trivial invariant subsets (\cite{Read}).
	
	\begin{proposition}
		Let $X$ be a non-zero Banach space. If $T\in \Lip_0(X,X)$ is odd, that is $T(-x)=-T(x)$ for every $x\in X$, then $\widehat T\colon \F(X)\longrightarrow \F(X)$ has a non-trivial invariant subspace.
	\end{proposition}
	
	\begin{proof}
		Define $S:=
		\overline{\mathrm{span}}
		\{\delta(x)-\delta(-x):x\in X\}
		\subset \F(X)$. Since $T$ is odd, we have
		$$
		\widehat T(\delta(x)-\delta(-x))
		=
		\delta(Tx)-\delta(-Tx)\in S,
		$$
		and therefore $\widehat T(S)\subset S$. We now check that $S$ is non-trivial. If $x\neq 0$, then $\delta(x)-\delta(-x)\neq 0$. Thus $S\neq \{0\}$. It remains to see that $S\neq \F(X)$. Let
		$$
		U:=\widehat{(-\mathrm{Id}_X)}\colon \F(X)\to \F(X).
		$$
		Then $U(\delta(x))=\delta(-x)$, and hence
		$$
		(I+U)(\delta(x)-\delta(-x))=0.
		$$
		Thus $S\subset \ker(I+U)$. On the other hand, $I+U$ is clearly non-zero. Hence $\ker(I+U)$ is a proper closed subspace of $\F(X)$, and consequently $S\neq \F(X)$. Therefore $S$ is a non-trivial invariant subspace for $\widehat T$.
	\end{proof}
	
	We conclude this section with a collection of elementary situations in which an operator $T_{f,e}$ admits a non-trivial invariant subspace. The following statement is meant to isolate several cases that will be excluded from the more delicate parts of the paper.
	
	\begin{proposition}\label{Prop: ISP}
		Let $M$ be an infinite separable metric space, let $f\in \Lip(M,M)$ and let $e\in M$. If one of the following assumptions is satisfied, then $T_{f,e}$ has a non-trivial invariant subspace:
		\begin{enumerate}[$(i)$]
			\item $f(M)$ is not dense in $M$;
			\item $f$ has a non-trivial and non-dense orbit;
			\item there exist $a\in M$ and $r>0$ such that $f(a)=a$, $\Lip\bigl(f|_{\overline B(a,r)}\bigr)\leq 1$,
			and $\{a\}\subsetneq \overline B(a,r)\subsetneq M$;
			\item there exist $a\in M$ and $r>0$ such that $f(a)=a$, $M\setminus B(a,r)\neq \emptyset$, and $d(f(y),a)\geq d(y,a)$ for all $y\in M\setminus B(a,r)$;
			\item $T_{f,e}$ is not injective. In particular, this holds if $f$ is not injective.
		\end{enumerate}
	\end{proposition}
	
	\begin{proof}
		We treat each case separately.
		\smallskip
		
		$(i)$ Put $A:=\overline{f(M)}$. Then $A$ is a proper closed subset of $M$. Moreover, $f(A)\subset A$. If $A$ contains at least two points, then Proposition~\ref{Prop: ISP2} applies and gives a non-trivial invariant subspace for $T_{f,e}$. If $A$ is a singleton, then $f$ is constant, and hence $T_{f,e}=0$. Since $M$ is infinite, $\F(M,e)$ is infinite-dimensional, and any one-dimensional subspace of $\F(M,e)$ is then a non-trivial invariant subspace.
		\smallskip
		
		$(ii)$ Let $x\in M$ be such that $A:=\overline{\Orb(x,f)}$ is a proper subset of $M$ and contains at least two points. Then $A$ is closed and satisfies $f(A)\subset A$. Hence Proposition~\ref{Prop: ISP2} shows that $T_{f,e}$ has a non-trivial invariant subspace.
		\smallskip
		
		$(iii)$ Let $A:=\overline B(a,r)$. For every $y\in A$, we have
		$$
		d(f(y),a)
		=
		d(f(y),f(a))
		\leq d(y,a)
		\leq r,
		$$
		and therefore $f(A)\subset A$. By assumption, $A$ is a proper closed subset of $M$ containing at least two points. Thus Proposition~\ref{Prop: ISP2} applies.
		\smallskip
		
		$(iv)$ Let
		$$
		A:=M\setminus B(a,r).
		$$
		Then $A$ is a non-empty proper closed subset of $M$. Moreover, $f(A)\subset A$. Indeed, if $y\in A$, then $d(y,a)\geq r$, so
		$$
		d(f(y),a)\geq d(y,a)\geq r.
		$$
		Thus $f(y)\in A$. If $A$ contains at least two points, Proposition~\ref{Prop: ISP2} applies. Otherwise, write $A=\{b\}$. Then $f(b)=b$, and hence $\{a,b\}$ is a proper closed subset of $M$ containing at least two points and satisfying
		$f(\{a,b\})\subset \{a,b\}$. Again, Proposition~\ref{Prop: ISP2} applies.
		\smallskip
		
		$(v)$ If $T_{f,e}$ is not injective and $T_{f,e}\neq 0$, then $\ker T_{f,e}$
		is a non-zero proper closed invariant subspace. If $T_{f,e}=0$, then every subspace of $\F(M,e)$ is invariant, and since $M$ is infinite, $\F(M,e)$ admits non-trivial closed subspaces.
		
		If $f$ is not injective, choose $x\neq y$ such that $f(x)=f(y)$. Then
		$$
		T_{f,e}(\delta(x)-\delta(y))
		=
		\delta(f(x))-\delta(f(y))
		=
		0,
		$$
		while $\delta(x)-\delta(y)\neq 0$. Hence $T_{f,e}$ is not injective.
	\end{proof}

	\section{Metric spaces ensuring NTIS for \texorpdfstring{$\widehat{f}$}{f hat}}
	
	In this section, $0$ will be the distinguished point of $M$, and we shall consider maps $f\in \Lip_0(M,M)$. The goal is to provide some partial positive answers to the invariant subspace problem for $\widehat f$. More precisely, we show that, under certain assumptions on the metric space $M$, every operator of the form $\widehat f$ admits a non-trivial invariant subspace.
	
	Let us point out that, as a map on $\F(M)$, $\widehat f$ always leaves the set
	$$
	\delta(M)=\{\delta(x):x\in M\}
	$$
	invariant. Of course, this set is not a linear subspace. The question is therefore whether this metric invariance can be upgraded to the existence of a non-trivial closed invariant linear subspace. We begin with a simple lemma which will be used repeatedly.
	
	\begin{lemma} \label{isolated}
		Let $M$ be an infinite pointed metric space with an isolated point $\overline{x}\in M$.
		Then, for any $f\in \Lip_0(M,M)$, $\widehat{f}$ has a non-trivial invariant subspace.
	\end{lemma}
	
	\begin{proof}
		By Proposition~\ref{Prop: ISP}, we may assume that $f$ is injective.
		
		If $\overline{x}=0$, then $M\setminus \{0\}$ is a closed subset of $M$. Moreover, since $f$ is injective and $f(0)=0$, we have $f(M\setminus\{0\})\subset M\setminus\{0\}$. Thus $M\setminus\{0\}$ is a non-trivial closed invariant subset for $f$, and Proposition~\ref{Prop: ISP2} applies.
		
		Assume now that $\overline{x}\neq 0$. If $\overline{x}$ is not a periodic point of $f$, then
		$$
		A:=\overline{\Orb(f(\overline{x}),f)}
		$$
		is a closed invariant subset which does not contain $\overline{x}$. Since $\overline{x}$ is isolated, this implies that $A\neq M$. Moreover, $A$ contains at least two points. Indeed, otherwise $f(\overline{x})$ would be a fixed point of $f$, contradicting the injectivity of $f$. Hence $A$ is a non-trivial closed invariant subset, and Proposition~\ref{Prop: ISP2} applies.
		
		If $\overline{x}$ is a periodic point of period at least two, then its periodic orbit is a non-trivial closed invariant subset of $M$, and again Proposition~\ref{Prop: ISP2} applies. Finally, if $\overline{x}$ is a fixed point, then $M\setminus\{\overline{x}\}$ is closed and invariant under $f$, by injectivity. Since $M$ is infinite, this is a non-trivial closed invariant subset, and the conclusion follows once more from Proposition~\ref{Prop: ISP2}.
	\end{proof}
	
	\begin{proposition} \label{Prop-Connected-Comp-fhat}
		Let $(M,d)$ be an infinite complete pointed metric space. Assume that $M$ has at least two connected components, and that one of them has non-empty interior.
		Then, for any $f\in \Lip_0(M,M)$, $\widehat{f}$ has a non-trivial invariant subspace.
	\end{proposition}
	
	\begin{proof}
		Let $f\in \Lip_0(M,M)$. By Proposition~\ref{Prop: ISP} and Lemma~\ref{isolated}, we may assume that $f$ is injective and that $M$ has no isolated points.
		
		Let $(C_i)_{i\in I}$ be the family of connected components of $M$, and let $C_0$ denote the connected component containing $0$. Since $f(C_0)$ is connected and contains $0$, we have $f(C_0)\subset C_0$. Therefore, if $C_0\neq \{0\}$, then $C_0$ is a non-trivial closed invariant subset of $M$, and Proposition~\ref{Prop: ISP2} yields the desired conclusion.
		
		We may therefore assume that $C_0=\{0\}$. Since $0$ is not isolated, this implies that $M$ has infinitely many connected components. Indeed, if there were only finitely many connected components, then $C_0$ would be clopen, and hence $0$ would be isolated.
		\smallskip
		
		Let $a\in I$ be such that $C_a$ has non-empty interior. Since $M$ has no isolated points, $C_a$ is not a singleton. Fix $x\in C_a$. We shall prove that $\Orb(x,f)$ is not dense in $M$. By Proposition~\ref{Prop: ISP}, this will give the desired conclusion.
		
		Suppose, towards a contradiction, that $\Orb(x,f)$ is dense in $M$. Since $C_a$ has non-empty interior, there exists $n\in \N$ such that $f^n(x)\in C_a$. Since $f$ is continuous, it maps a connected component into a connected component, and as it returns to $C_a$, it visits only finitely $C_i$'s. Thus there exist $i_1,\ldots,i_m\in I$ with $m\leq n$ such that
		$$
		\Orb(x,f)\subset \bigcup_{k=1}^m C_{i_k}.
		$$
		Since each connected component is closed, the finite union $\bigcup_{k=1}^m C_{i_k}$ is closed. It is also a proper subset of $M$, because $M$ has infinitely many connected components. Hence
		$$
		\overline{\Orb(x,f)}\subset \bigcup_{k=1}^m C_{i_k}\neq M,
		$$
		contradicting the density of $\Orb(x,f)$.
	\end{proof}
	
	The next definition comes from \cite{G}.
	
	\begin{definition}
		Let $f : M \to M$ be a continuous map. We say that $x\in M$ is almost periodic for $f$ if, for every $\delta>0$, there exists $N\in \N$ such that for every $n\in \N$, there is $k\in\{n+1,\ldots,n+N\}$ satisfying
		$$
		d(f^k(x),x)\leq \delta.
		$$
	\end{definition}
	
	The following lemma and its proof are inspired by \cite[Lemma 4]{G}.
	
	\begin{lemma} \label{lemma:AlmostP}
		Let $M$ be a metric space and let $f : M \to M$ be continuous. Suppose that there are $\overline{x}\in M$ and $\varepsilon>0$ such that $\overline{x}$ is not isolated, $\overline{B}(\overline{x},\varepsilon)$ is compact, and every point in $\overline{B}(\overline{x},\varepsilon)$ has a dense orbit.
		Then $\overline{x}$ is an almost periodic point of $f$.
	\end{lemma}
	
	\begin{proof}
		Suppose, towards a contradiction, that $\overline{x}$ is not almost periodic for $f$. Then there exists $\delta>0$ such that, for every $N\in \N$, there is $n_N\in \N$ such that
		$$
		d(f^k(\overline{x}),\overline{x})>\delta,
		\qquad \forall k\in\{n_N+1,\ldots,n_N+N\}.
		$$
		Replacing $\delta$ by a smaller positive number if necessary, we may assume that $\delta<\varepsilon$.
		
		For each $N\in \N$, let $r_N$ be the largest integer $r\leq n_N$ such that
		$$
		f^r(\overline{x})\in B(\overline{x},\delta).
		$$
		Such an integer exists, since $0\leq n_N$ and $f^0(\overline{x})=\overline{x}\in B(\overline{x},\delta)$. By the maximality of $r_N$ and by the choice of $n_N$, we have
		$$
		d(f^{r_N+j}(\overline{x}),\overline{x})\geq \delta,
		\qquad \forall j\in\{1,\ldots,N\}.
		$$
		Moreover,
		$$
		f^{r_N}(\overline{x})\in \overline{B}(\overline{x},\delta)\subset \overline{B}(\overline{x},\varepsilon).
		$$
		
		By compactness of $\overline{B}(\overline{x},\varepsilon)$, passing to a subsequence if necessary, we may assume that
		$$
		f^{r_N}(\overline{x})\longrightarrow \ell
		$$
		for some $\ell\in \overline{B}(\overline{x},\varepsilon)$. Fix $j\geq 1$. For all $N\geq j$, we have
		$$
		d(f^j(f^{r_N}(\overline{x})),\overline{x})\geq \delta.
		$$
		Passing to the limit and using the continuity of $f^j$, we obtain
		$$
		d(f^j(\ell),\overline{x})\geq \delta.
		$$
		This holds for every $j\geq 1$. Hence the orbit of $\ell$ does not meet the open ball $B(\overline{x},\delta)$ after time $0$, which contradicts the assumption that every point in $\overline{B}(\overline{x},\varepsilon)$ has a dense orbit.
	\end{proof}
	
	We now arrive at the main result of this section.
	
	\begin{theorem} \label{Thm:CompactBall}
		Let $M$ be an infinite pointed metric space. Suppose that there are $\overline{x}\in M$ and $\varepsilon>0$ such that $\overline{B}(\overline{x},\varepsilon)$ is compact.
		Then, for any $f\in \Lip_0(M,M)$, $\widehat{f}$ has a non-trivial invariant subspace.
	\end{theorem}
	
	\begin{proof}
		Let $f\in \Lip_0(M,M)$. By Proposition~\ref{Prop: ISP}, we may assume that every non-zero point of $M$ has a dense orbit and that $L:=\Lip(f)>1$. By Lemma~\ref{isolated}, we may also assume that $M$ has no isolated points.
		
		By replacing $\overline{x}$ with another point of the compact ball and by shrinking $\varepsilon>0$ if necessary, we may assume that $0\notin \overline{B}(\overline{x},\varepsilon)$.
		
		It follows from our reductions that every point of $\overline{B}(\overline{x},\varepsilon)$ has a dense orbit. Hence, by Lemma~\ref{lemma:AlmostP}, the point $\overline{x}$ is almost periodic for $f$. We shall prove that this is impossible.
		
		Put
		$$
		\rho:=\frac{d(\overline{x},0)}{2}>0.
		$$
		Let $N\in \N$. Since the orbit of $\overline{x}$ is dense in $M$, there exists $n\in \N$ such that
		$$
		d(f^n(\overline{x}),0)<L^{-N}\rho.
		$$
		Since $f(0)=0$ and $f$ is $L$-Lipschitz, for every $k\in\{0,\ldots,N\}$ we have
		$$
		d(f^{n+k}(\overline{x}),0)
		=
		d(f^k(f^n(\overline{x})),f^k(0))
		\leq
		L^k d(f^n(\overline{x}),0)
		<
		L^{k-N}\rho
		\leq \rho.
		$$
		Therefore, for every $k\in\{0,\ldots,N\}$,
		$$
		d(f^{n+k}(\overline{x}),\overline{x})
		\geq
		d(\overline{x},0)-d(f^{n+k}(\overline{x}),0)
		>
		2\rho-\rho
		=
		\rho.
		$$
		In particular,
		$$
		d(f^{n+k}(\overline{x}),\overline{x})>\rho,
		\qquad \forall k\in\{1,\ldots,N\}.
		$$
		Since this can be done for every $N\in\N$, the point $\overline{x}$ is not almost periodic, a contradiction.
	\end{proof}
	
	\begin{remark}
		Without such a compactness assumption, the situation is quite different. For instance, Read's examples on $\ell_1$ show that one may have a bounded linear operator for which every non-zero orbit is dense. In these examples, no compactness mechanism is available to force the existence of almost periodic points.
	\end{remark}

	\section{Metric spaces ensuring NTIS for \texorpdfstring{$T_{f,e}$}{Tf,e}}
	
	We now turn to the study of invariant subspaces for $T_{f,e}$, with the same general philosophy as in the previous section: we look for metric properties of $M$ which ensure the existence of a NTIS for every $f\in \Lip(M,M)$. We start with the following lemma, which is the natural analogue of Lemma~\ref{isolated}.
	
	\begin{lemma}\label{Prop:IsolatedT}
		Let $M$ be an infinite metric space with an isolated point.
		Then, for every Lipschitz map $f : M \to M$ and every $e\in M$, $T_{f,e}$ has a non-trivial invariant subspace.
	\end{lemma}
	
	\begin{proof}
		Thanks to Proposition~\ref{prop:conjugate}, we may assume that $e$ is an isolated point of $M$. By Proposition~\ref{Prop: ISP}, we may also assume that $f$ is injective.
		
		If $f(e)=e$, then $T_{f,e}=\widehat f : \F(M,e)\to \F(M,e)$, and Lemma~\ref{isolated} gives the result. We may therefore assume that $f(e)\neq e$.
		
		Assume first that there exists $n\geq 2$ such that $f^n(e)=e$. Then
		$$
		T_{f,e}^n(\delta(f(e)))
		=
		\delta(f^{n+1}(e))-\delta(f^n(e))
		=
		\delta(f(e)).
		$$
		Thus $\delta(f(e))$ is a periodic point of $T_{f,e}$. Hence
		$$
		\overline{\lspan}\Orb(\delta(f(e)),T_{f,e})
		$$
		is a finite-dimensional non-zero invariant subspace. Since $M$ is infinite, $\F(M,e)$ is infinite-dimensional, so this subspace is proper.
		
		Assume now that $e\notin \Orb(f(e),f)$. Since $e$ is isolated, $e\notin \overline{\Orb(f(e),f)}$.
		Set $A:=\overline{\Orb(f(e),f)}$. Then $A$ is a proper closed invariant subset of $M$ for $f$. Moreover, $A$ contains at least two points. Indeed, if $A$ were a singleton, then $f(e)$ would be a fixed point of $f$, contradicting the injectivity of $f$ and the fact that $f(e)\neq e$. Thus $A$ is a non-trivial closed invariant subset, and Proposition~\ref{Prop: ISP2} applies.
	\end{proof}
	
	We now continue with the next proposition, which should be compared with Proposition~\ref{Prop-Connected-Comp-fhat}.
	
	\begin{proposition}\label{Prop:InfiniteConnectedComponent}
		Let $M$ be a complete metric space and let $f : M \to M$ be a Lipschitz map. Assume that $M$ has infinitely many connected components and that one of them has non-empty interior.
		Then, for every $e\in M$, $T_{f,e}$ has a non-trivial invariant subspace.
	\end{proposition}
	
	\begin{proof}
		If $f$ has a fixed point $a\in M$, then $T_{f,e}$ is conjugate to $T_{f,a}=\widehat f_a$. Hence Proposition~\ref{Prop-Connected-Comp-fhat} gives the conclusion. We may therefore assume that $f$ has no fixed point.
		
		Let $C_a$ be a connected component with non-empty interior, and let $x\in C_a$. If $C_a$ is finite, then $C_a$ is a singleton, hence $M$ has an isolated point, and Lemma~\ref{Prop:IsolatedT} applies. Thus we may assume that $C_a$ is infinite.
		
		If $\Orb(x,f)$ is not dense in $M$, then Proposition~\ref{Prop: ISP} applies, since the orbit is non-trivial.
		
		Finally, if $\Orb(x,f)$ is dense in $M$, we obtain a contradiction as in the proof of Proposition \ref{Prop-Connected-Comp-fhat}. 
	\end{proof}
	
	We conclude this first part of the section with the following theorem, which should be compared with Theorem~\ref{Thm:CompactBall}.
	
	\begin{theorem}\label{Thm-Tfe-Compactball}
		Let $M$ be a complete non-compact metric space, and let $f\in \Lip(M,M)$. Assume that there are $x\in M$ and $r>0$ such that $\overline{B}(x,r)$ is compact.
		Then, for every $e\in M$, $T_{f,e}$ admits a non-trivial invariant subspace.
	\end{theorem}
	
	\begin{proof}
		If $f$ has a fixed point $a\in M$, then $T_{f,e}$ is conjugate to $T_{f,a}=\widehat f_a$, and Theorem~\ref{Thm:CompactBall} applies. We may therefore assume that $f$ has no fixed point.
		
		By Proposition~\ref{Prop: ISP}, we may assume that every orbit is dense in $M$. Indeed, since $f$ has no fixed point, every orbit contains at least two points. By Lemma~\ref{Prop:IsolatedT}, we may also assume that $M$ has no isolated points.
		
		Since every point of $\overline{B}(x,r)$ has a dense orbit, Lemma~\ref{lemma:AlmostP} implies that $x$ is almost periodic for $f$. Applying almost periodicity with $\delta=r$, there exists $N\in \N$ such that, for every $n\in \N$, there is $k\in\{n+1,\ldots,n+N\}$ satisfying $f^k(x)\in \overline{B}(x,r)$. We claim that
		$$
		M\subseteq \bigcup_{j=0}^{N-1} f^j(\overline{B}(x,r)).
		$$
		Let $m\geq N$. Applying the previous assertion with $n=m-N$, there exists $i\in \{m-N+1,\ldots,m\}$
		such that $f^i(x)\in \overline{B}(x,r)$. If $j=m-i$, then $j\in\{0,\ldots,N-1\}$ and
		$$
		f^m(x)=f^j(f^i(x))\in f^j(\overline{B}(x,r)).
		$$
		Hence
		$$
		\{f^m(x):m\geq N\}
		\subseteq
		\bigcup_{j=0}^{N-1} f^j(\overline{B}(x,r)).
		$$
		Since $\Orb(x,f)$ is dense and $M$ has no isolated points, the tail $\{f^m(x):m\geq N\}$ is still dense in $M$. Therefore
		$$
		M
		\subseteq
		\bigcup_{j=0}^{N-1} f^j(\overline{B}(x,r)).
		$$
		The right-hand side is compact, being a finite union of compact sets. Thus $M$ is compact, a contradiction.
	\end{proof}
	
	As one can see, the techniques used in the previous section to treat the compact-ball case or the disconnected case for $\widehat f$ do not immediately give results of the same strength for the operators $T_{f,e}$. To obtain sharper statements, one needs additional ideas. This is what we do next, although we do not manage to cover all the cases treated in the previous section.
	
	We shall use \cite{AP23}, where Aliaga and Pernecká systematically study the links between measures on a metric space and elements of the corresponding Lipschitz-free space. In particular, we shall use the following consequence of \cite[Proposition~4.4]{AP23}: every finite Borel measure $\mu$ on $M$ with compact support induces an element of $\F(M,e)$, denoted by
	$$
	\mathcal L\mu=\int_M \delta(x)\,d\mu(x),
	$$
	where the integral is understood in the Bochner sense. Indeed, the condition appearing in \cite[Proposition~4.4]{AP23} is that $\mu$ must have finite first moment, namely
	$$
	\int_M d(x,e)\,d|\mu|(x)<+\infty.
	$$
	This condition is automatically satisfied when $\mu$ has compact support, since the function $x\mapsto d(x,e)$ is continuous and therefore bounded on $\supp(\mu)$.
	
	Let us also comment on the assumption $\mu(\{e\})=0$ appearing in \cite[Proposition~4.4]{AP23}, where the basepoint is denoted by $0$. This is essentially a normalization. Indeed, every function in $\Lip_0(M,e)$ vanishes at $e$, and therefore any mass placed at $e$ is invisible to the induced functional. Equivalently, $\delta(e)=0$ in $\F(M,e)$. Thus a measure $\mu$ may be replaced by
	$$
	\mu-\mu(\{e\})\delta_e
	$$
	without changing the induced element of $\F(M,e)$.
	
	For positive finite measures with compact support, we shall also use the fact that the support of $\mathcal L\mu$ agrees with the usual support of $\mu$ (see \cite[Proposition~4.8]{AP23}) and there is uniqueness of the representing measure ($\mathcal{L}\mu = \mathcal{L} \nu \implies \mu = \nu$, see \cite[Proposition~4.9]{AP23}). As usual, if $\mu$ is a measure on a compact metric space $K$ and $f:K\to K$ is continuous, then $f_\#\mu$ denotes the pushforward of $\mu$ by $f$.
	
	\begin{proposition}\label{prop:push-forward-T}
		Let $K$ be a compact metric space, let $e\in K$, and let $f\colon K\to K$ be Lipschitz. For every finite positive Borel measure $\mu$ on $K$, one has
		$$
		T_{f,e}(\mathcal L\mu)
		=
		\mathcal L(f_\#\mu)-\mu(K)\delta(f(e)).
		$$
	\end{proposition}
	
	\begin{proof}
		By linearity and continuity of the Bochner integral,
		$$
		\begin{aligned}
			T_{f,e}(\mathcal L\mu)
			&=
			T_{f,e}\left(\int_K \delta(x)\,d\mu(x)\right) 
			=
			\int_K T_{f,e}(\delta(x))\,d\mu(x) \\
			&=
			\int_K \bigl(\delta(f(x))-\delta(f(e))\bigr)\,d\mu(x) 
			=
			\mathcal L(f_\#\mu)-\mu(K)\delta(f(e)).
		\end{aligned}
		$$
	\end{proof}
	
	
	\begin{proposition}\label{thm:two-measures-T}
		Let $K$ be an infinite compact metric space, let $f\colon K\to K$ be Lipschitz, and let $e\in K$. Assume that $f$ admits two distinct invariant Borel probability measures. Then $T_{f,e}$ has a non-trivial invariant closed subspace. In fact, $T_{f,e}$ has a non-zero fixed point.
	\end{proposition}
	
	\begin{proof}
		Let $\mu$ and $\nu$ be two distinct invariant Borel probability measures on $K$, and define
		$$
		\eta:=\mathcal L\mu-\mathcal L\nu.
		$$
		Since $\mu \neq \nu$, \cite[Proposition~2.9]{AP23} gives that $\eta \neq 0$. Using Proposition~\ref{prop:push-forward-T} and the invariance of $\mu$ and $\nu$, we obtain
		$$
		\begin{aligned}
			T_{f,e}(\eta)
			&=
			T_{f,e}(\mathcal L\mu)-T_{f,e}(\mathcal L\nu) \\
			&=
			\bigl(\mathcal L(f_\#\mu)-\delta(f(e))\bigr)
			-
			\bigl(\mathcal L(f_\#\nu)-\delta(f(e))\bigr) \\
			&=
			\mathcal L\mu-\mathcal L\nu
			=
			\eta.
		\end{aligned}
		$$
		Thus $\eta$ is a non-zero fixed point of $T_{f,e}$.
	\end{proof}
	
	The next proposition is the first step toward an analogue of Proposition~\ref{Prop-Connected-Comp-fhat} in the compact setting.
	
	\begin{proposition}\label{prop:periodic-closed-set-T} 
		Let $K$ be an infinite compact metric space, let $f\in \Lip(K,K)$, and let $e\in K$. Assume that there exist a non-empty closed set $C\subset K$ and an integer $n\geq 2$ such that the sets $C,\ f(C),\ \dots,\ f^{n-1}(C)$ are pairwise disjoint and $f^n(C)=C$. Then $T_{f,e}$ has a non-trivial invariant closed subspace. 
	\end{proposition}
	
	\begin{proof}
		Suppose first that $\K=\C$. By the Krylov--Bogolyubov theorem (see, e.g. \cite[Corollary 6.9.1]{Walters}), applied to the continuous map $f^n:C\to C$, there exists an $f^n$-invariant Borel probability measure $\mu_0$ on $C$. For $j=0,\dots,n-1$, define
		$$
		\mu_j:=(f^j)_\#\mu_0.
		$$
		Then each $\mu_j$ is a probability measure supported on $f^j(C)$, and
		$$
		f_\#\mu_j=\mu_{j+1},
		\qquad j=0,\dots,n-1,
		$$
		where indices are understood modulo $n$.
		
		Let $\omega=e^{2\pi i/n}$ and define
		$$
		\eta:=\sum_{j=0}^{n-1}\omega^j\mathcal L\mu_j\in \F(K,e).
		$$
		Using Proposition~\ref{prop:push-forward-T}, and the fact that each $\mu_j$ is a probability measure, we obtain
		$$
		\begin{aligned}
			T_{f,e}(\eta)
			&=
			\sum_{j=0}^{n-1}\omega^j T_{f,e}(\mathcal L\mu_j) \\
			&=
			\sum_{j=0}^{n-1}\omega^j\bigl(\mathcal L(f_\#\mu_j)-\delta(f(e))\bigr) \\
			&=
			\sum_{j=0}^{n-1}\omega^j\mathcal L\mu_{j+1}
			-
			\left(\sum_{j=0}^{n-1}\omega^j\right)\delta(f(e)).
		\end{aligned}
		$$
		Since $\sum_{j=0}^{n-1}\omega^j=0$, it follows that
		$$
		T_{f,e}(\eta)
		=
		\sum_{j=0}^{n-1}\omega^j\mathcal L\mu_{j+1}
		=
		\omega^{-1}\eta.
		$$
		
		It remains to check that $\eta\neq 0$. Since the compact sets $C,\ f(C),\ \dots,\ f^{n-1}(C)$
		are pairwise disjoint, there exists a Lipschitz function $g:K\to\C$ such that $g=1$ on $C$ and $g=0$ on $f^j(C)$ for every $j=1,\dots,n-1$. Set $h:=g-g(e)$. Then $h\in \Lip_0(K,e)$, and
		$$
		\begin{aligned}
			\langle \eta,h\rangle
			&=
			\sum_{j=0}^{n-1}\omega^j\int_K (g-g(e))\,d\mu_j \\
			&=
			\sum_{j=0}^{n-1}\omega^j\int_K g\,d\mu_j
			-
			g(e)\sum_{j=0}^{n-1}\omega^j\mu_j(K) \\
			&=
			\sum_{j=0}^{n-1}\omega^j\int_K g\,d\mu_j
			=
			1.
		\end{aligned}
		$$
		Thus $\eta\neq 0$. Hence $\mathbb C\eta$ is a non-trivial invariant subspace for $T_{f,e}$.
		
		Assume now that $\K=\R$. If $n=2$, then $\omega=-1$, and the above construction gives $\eta=\mathcal L\mu_0-\mathcal L\mu_1\in \F(K,e)$,
		with $T_{f,e}(\eta)=-\eta$. Hence $\R\eta$ is a non-trivial invariant subspace.
		
		If $n\geq 3$, we apply the complex case to the complexification of $\F(K,e)$. Writing the corresponding eigenvector as $\eta=u+iv$, with $u,v\in \F(K,e)$, the real subspace $\lspan_{\R}\{u,v\}$ is non-zero, finite-dimensional, and invariant under $T_{f,e}$. Since $K$ is infinite, $\F(K,e)$ is infinite-dimensional, so this subspace is proper.
	\end{proof}
	
	In the noncompact setting we have the following partial result.
	
	\begin{proposition}
		\label{prop: slightly noncompact}
		Let $M$ be an infinite complete metric space and let $e\in M$. 
		Assume that $M$ has at least two connected components, one of them with at least two points, and that there is a closed set $S\subsetneq M$ such that \[\mathrm{dist}(S,M\setminus S):=\inf\{d(x,y):~x\in S,~y\in M\setminus S\} >0.\] 
		Then, $T_{f,e}$ admits a nontrivial invariant subspace for any Lipschitz map ${f:M\to M}$.   
	\end{proposition}
	
	\begin{proof}
		Let $C$ be the connected component with at least $2$ points and fix ${\{x_1,x_2\}\in C}$, with $x_1\neq x_2$.
		Set $\mu\in \mathcal{F}(M,e)$ as the vector defined by
		\[\mu:=\delta(x_1)-\delta(x_2).\]
		Thus, for any $n\in\N$, we have $T^n_{f,e}(\mu):=\delta(f^n(x_1))-\delta(f^n(x_2))$.
		Thanks to the continuity of $f$, we deduce that, for any $n\geq 0$, $f^n(x_1)$ and $f^n(x_2)$ belong to the same connected component. Since $\mathrm{dist}(S,M\setminus S)>0$, both $S$ and $M\setminus S$ are clopen, and hence every connected component of $M$ is contained in one of these two sets.
		Consider now $\phi:M\to \K$ be the function defined as $\phi(S)=\{0\}$ and $\phi (M\setminus S)=\{1\}$ if $e\in S$, or $\phi(S)=\{1\}$ and $\phi (M\setminus S)=\{0\}$ if $e\in M\setminus S$. 
		Since $\mathrm{dist}(S,M\setminus S)>0$, $\phi\in \mathrm{Lip}_0(M)$. 
		Moreover, observe that 
		\[\phi(T^n_{f,e}(\mu))=0,\quad\text{for all }n\geq 0.\]
		Thus, $E:=\overline{\mathrm{span}}\,\mathrm{Orb}(\mu,T_{f,e})\subset \ker(\phi)$ is a nontrivial invariant subspace for $T_{f,e}$.
	\end{proof}
	
	\begin{corollary}\label{cor:compact-components-T}
		Let $K$ be an infinite compact metric space with at least two connected components. Assume that one connected component has at least two points. Then, $T_{f,e}$ admits a nontrivial invariant subspace for any Lipschitz map ${f:K\to K}$. 
	\end{corollary}
	
	\begin{proof}
		Let $C\subset K$ be the connected component with at least two points. 
		Since $K$ is not connected, there is a nonempty clopen set $S\subsetneq K$. By compactness of $K$, $\mathrm{dist}(S,K\setminus S)>0$. Thus, Proposition~\ref{prop: slightly noncompact} applies.
	\end{proof}
	
	The results Lemma~\ref{Prop:IsolatedT}, Proposition~\ref{Prop:InfiniteConnectedComponent}, Proposition~\ref{prop: slightly noncompact} and
	Corollary~\ref{cor:compact-components-T} point towards the idea that the operators $T_{f,e}$ always admit a non-trivial invariant subspace whenever the ambient space $M$ is disconnected. However, there are still some remaining open cases in this direction, for instance: 
	$M$ is compact and every connected component is a singleton or
	$M$ is noncompact and it cannot be decomposed in two clopen sets in the sense of Proposition~\ref{prop: slightly noncompact} (therefore, there is no nontrivial piecewise constant Lipschitz function).

	\section{Final remarks and linear dynamics}
	
	The invariant subspace problem is closely related to questions from linear dynamics. Indeed, a bounded operator $T:X\to X$ has no non-trivial invariant subspace precisely when every non-zero vector is cyclic for $T$. Thus, although the invariant subspace problem is not itself a problem about hypercyclicity, it naturally leads one to study the size and structure of linear orbits. Conversely, linear dynamics provides a useful language for comparing the metric dynamics of a map $f:M\to M$ with the linear dynamics of the induced operators $\widehat f$ and $T_{f,e}$.
	
	In this final section, we collect a few observations in this direction. Our starting point is \cite[Proposition~2.7]{ACP20}, which shows that, for $f\in \Lip_0(M,M)$, the orbit $\Orb(x,f)$ is dense in $M$ if and only if $\delta(x)$ is a cyclic vector for $\widehat f$. 
	It is therefore natural to ask to what extent \cite[Proposition~2.7]{ACP20} extends from free linearizations $\widehat f$ to the more general operators $T_{f,e}$.
	
	The defining formula for $T_{f,e}$ already suggests that the situation is more delicate. Namely,
	$$
	T_{f,e}^n(\delta(x))=\delta(f^n(x))-\delta(f^n(e)),
	\qquad \forall n\in \N.
	$$
	Thus the orbit of $\delta(x)$ under $T_{f,e}$ is governed simultaneously by the two metric orbits $\Orb(x,f)$ and $\Orb(e,f)$. This should be contrasted with the case of $\widehat f$, where the orbit of $\delta(x)$ depends only on $\Orb(x,f)$.
	
	Let us first record two elementary observations. If $\delta(x)$ is cyclic for $T_{f,e}$, then $\Orb(x,f)\cup\Orb(e,f)$ must be dense in $M$. Indeed,
	$$
	\Orb(\delta(x),T_{f,e})
	\subset
	\overline{\lspan}\,\delta\bigl(\Orb(x,f)\cup\Orb(e,f)\bigr).
	$$
	On the other hand, $\delta(x)$ may be cyclic for $T_{f,e}$ even when $\Orb(x,f)$ is not dense. The following finite example illustrates this phenomenon.
	
	\begin{example}
		Let $M:=\{0,1,\ldots,N\}\subset \mathbb N\cup\{0\}$, endowed with the usual metric, and define $f:M\to M$ by $f(0)=0$, $f(N)=0$, and $f(n)=n+1$ for $1\leq n\leq N-1$.
		Choosing $e=1$, we have $T^0_{f,1}(\delta(0))=\delta(0)$ and, for $1\leq n\leq N-1$,
		$$
		T^n_{f,1}(\delta(0))
		=
		\delta(f^n(0))-\delta(f^n(1))
		=
		\delta(0)-\delta(n+1).
		$$
		Moreover, $T^N_{f,1}(\delta(0))=0$. Therefore
		$$
		\Orb(\delta(0),T_{f,1})
		=
		\{0,\delta(0),\delta(0)-\delta(2),\ldots,\delta(0)-\delta(N)\}.
		$$
		It follows that
		$$
		\mathrm{span}\Orb(\delta(0),T_{f,1})=\F(M,1).
		$$
		Thus $\delta(0)$ is a cyclic vector for $T_{f,1}$, although $\Orb(0,f)=\{0\}$ is not dense in $M$.
	\end{example}
	
	We now give a second example, this time on the circle, where the cyclicity of the vectors $\delta(x)$ can be completely described. Let $\mathbb T=\mathbb R/\mathbb Z$ be equipped with its usual arclength metric. By symmetry, we take $0$ as the basepoint. Recall that if $f$ is an irrational rotation of $\mathbb T$, then every orbit $\Orb(x,f)$ is dense in $\mathbb T$. Nevertheless, not every vector $\delta(x)$ is cyclic for $T_{f,0}$.
	We shall use the following simple consequence of the Hahn--Banach theorem.
	
	\begin{lemma}\label{lemma:delta_x cyclic}
		Let $x$ be an element of a pointed metric space $(M,e)$, and let $f\in \Lip(M,M)$. Then the following assertions are equivalent:
		\begin{enumerate}[$(i)$]
			\item $\delta(x)$ is not a cyclic vector for $T_{f,e}$;
			\item there exists $\varphi\in \Lip_0(M,e)\setminus\{0\}$ such that
			$$
			\langle \varphi,\gamma\rangle=0,
			\qquad \forall \gamma\in \lspan\Orb(\delta(x),T_{f,e});
			$$
			\item there exists $\varphi\in \Lip_0(M,e)\setminus\{0\}$ such that
			$$
			\varphi(f^n(x))=\varphi(f^n(e)),
			\qquad \forall n\in \N\cup\{0\}.
			$$
		\end{enumerate}
	\end{lemma}

	\begin{proposition}\label{prop:cyclic}
		Let $f:t\in \mathbb T\mapsto t+r\in \mathbb T$ be an irrational rotation, and let $s\in \mathbb T$. Then $\delta(s)$ is a cyclic vector for $T_{f,0}$ if and only if $s$ is irrational.
	\end{proposition}
	
	\begin{proof}
		Suppose first that $s$ is irrational. Let $x\in\mathbb T$ and let $\varepsilon>0$. Since the set $\{Ns:N\in\N\}$ is dense in $\mathbb T$, there exists $N\in\N$ such that, writing $w=Ns$, one has $d(w,x)<\varepsilon/2$.
		Moreover,
		$$
		\delta(w)
		=
		\sum_{j=1}^N\bigl(\delta(js)-\delta((j-1)s)\bigr).
		$$
		For each $j\in\{1,\dots,N\}$, since the orbit of $0$ under the irrational rotation $f$ is dense, we may choose $n_j\in\N$ such that
		$$
		d(n_jr,(j-1)s)<\frac{\varepsilon}{4N}.
		$$
		By translation-invariance of the metric on $\mathbb T$, this also gives
		$$
		d(s+n_jr,js)<\frac{\varepsilon}{4N}.
		$$
		Hence, for each $j\in\{1,\dots,N\}$,
		$$
		\begin{aligned}
			&\left\|\delta(js)-\delta((j-1)s)
			-\bigl(\delta(f^{n_j}(s))-\delta(f^{n_j}(0))\bigr)\right\| \\
			&\qquad\leq
			\|\delta(js)-\delta(f^{n_j}(s))\|
			+
			\|\delta((j-1)s)-\delta(f^{n_j}(0))\| \\
			&\qquad<
			\frac{\varepsilon}{2N}.
		\end{aligned}
		$$
		Now set $\gamma:=\sum_{j=1}^N T_{f,0}^{n_j}(\delta(s))$.
		Then $\gamma\in \lspan\Orb(\delta(s),T_{f,0})$ and
		$$
		\|\delta(w)-\gamma\|<\varepsilon/2.
		$$
		Therefore
		$$
		\|\delta(x)-\gamma\|
		\leq
		\|\delta(x)-\delta(w)\|+\|\delta(w)-\gamma\|
		<\varepsilon.
		$$
		Since $\overline{\lspan}\{\delta(x):x\in\mathbb T\}=\F(\mathbb T,0)$, we conclude that $\delta(s)$ is cyclic for $T_{f,0}$.
		\medskip
		
		Conversely, if $s=0$, then $\delta(s)=0$ and hence $\delta(s)$ is not cyclic. Assume now that $s$ is rational and non-zero. Write $s=p/q$ with $p,q\in\mathbb N$ and $\gcd(p,q)=1$. Let $\varphi\in \Lip_0(\mathbb T,0)$ be given by
		$$
		\varphi(t)=\sin(2\pi qt).
		$$
		Then, for every $n\in\N\cup\{0\}$,
		$$
		\begin{aligned}
			\varphi(f^n(s))-\varphi(f^n(0))
			&=
			\sin(2\pi q(s+nr))-\sin(2\pi qnr)\\
			&=
			\sin(2\pi p+2\pi qnr)-\sin(2\pi qnr)\\
			&=
			0.
		\end{aligned}
		$$
		By Lemma~\ref{lemma:delta_x cyclic}, $\delta(s)$ is not cyclic.
	\end{proof}
	
	The preceding examples show that, for arbitrary choices of $e$, \textit{neither} of the naive implications
	$$
	\Orb(x,f)\text{ is dense in }M
	\quad\Longrightarrow\quad
	\delta(x)\text{ is cyclic for }T_{f,e}
	$$
	and
	$$
	\delta(x)\text{ is cyclic for }T_{f,e}
	\quad\Longrightarrow\quad
	\Orb(x,f)\text{ is dense in }M
	$$
	holds in general. However, if the basepoint is chosen in a suitable way, then the two notions coincide.
	
	\begin{proposition}
		Let $f\in \Lip(M,M)$, let $x\in M$, and set $e:=f(x)$. Then $\Orb(x,f)$ is dense in $M$ if and only if $\delta(x)$ is a cyclic vector for $T_{f,e}$.
	\end{proposition}
	
	\begin{proof}
		Suppose first that $\Orb(x,f)$ is dense in $M$. Let $\varphi\in \Lip_0(M,e)$ be such that
		$$
		\varphi(f^n(x))=\varphi(f^n(e)),
		\qquad \forall n\in\N\cup\{0\}.
		$$
		Since $e=f(x)$, this means that
		$$
		\varphi(f^n(x))=\varphi(f^{n+1}(x)),
		\qquad \forall n\in\N\cup\{0\}.
		$$
		Hence $\varphi$ is constant on $\Orb(x,f)$. Since this orbit is dense in $M$, $\varphi$ is constant on $M$. As $\varphi(e)=0$, we get $\varphi=0$. By Lemma~\ref{lemma:delta_x cyclic}, $\delta(x)$ is cyclic for $T_{f,e}$.
		\smallskip
		
		Conversely, assume that $\delta(x)$ is cyclic for $T_{f,e}$. Since $e=f(x)$, one has
		$$
		\overline{\lspan}\Orb(\delta(x),T_{f,e})
		=
		\overline{\lspan}\{\delta(f^n(x))-\delta(f^{n+1}(x)):n\geq 0\}
		\subset
		\F\bigl(\overline{\Orb(x,f)},e\bigr).
		$$
		Therefore, cyclicity of $\delta(x)$ implies
		$$
		\F(M,e)=\F\bigl(\overline{\Orb(x,f)},e\bigr).
		$$
		This is possible only if $\overline{\Orb(x,f)}=M$. Hence $\Orb(x,f)$ is dense in $M$.
	\end{proof}
	
	\begin{remark}
		The proof of the converse implication works more generally whenever $e\in\Orb(x,f)\setminus \{x\}$. However, the direct implication may fail for other choices of $e\in\Orb(x,f)$. For instance, suppose that $M$ is the union of two disjoint subsets $M_1$ and $M_2$ at positive distance from each other, and that $f$ alternates between them. If both $x$ and $e$ belong to $M_1$, then the Lipschitz function $\varphi$ defined by
		$$
		\varphi|_{M_1}=0,
		\qquad
		\varphi|_{M_2}=1
		$$
		satisfies $\varphi(f^n(x))=\varphi(f^n(e))$ for every $n\in\N\cup\{0\}$. Thus, by Lemma~\ref{lemma:delta_x cyclic}, $\delta(x)$ is not cyclic for $T_{f,e}$, even though the orbit of $x$ under $f$ may be dense in $M$. Note also that if, for every $k\in\N$, the subsequence
		$\{f^{nk}(x):n\geq 0\}$ is dense in $M$, then $\delta(x)$ is a cyclic vector for $T_{f,e}$ for every $e\in\Orb(x,f)\setminus\{x\}$. This condition is satisfied when $T$ is a hypercyclic linear operator and $x$ is a hypercyclic vector for~$T$ (Ansari’s Theorem).
	\end{remark}
	
	A continuous map $f:M\to M$ is said to be \emph{topologically transitive} if, for every pair of non-empty open sets $U,V\subset M$, there exists $n\in\N\cup\{0\}$ such that
	$$
	f^n(U)\cap V\neq\emptyset.
	$$
	Recall that, on a separable Banach space, a bounded operator $T:X\to X$ is hypercyclic if and only if it is topologically transitive; see, for instance, \cite[Proposition~1.15]{GEPM}. A stronger notion is weak mixing: a continuous map $f:M\to M$ is \emph{weakly mixing} if $f\times f$ is topologically transitive on $M\times M$.
	
	In \cite{DR}, De La Rosa and Read proved that there exist hypercyclic operators which are not weakly mixing. The next proposition shows that, when $M$ is bounded, such operators cannot arise as $T_{f,e}$.
	
	\begin{proposition}
		Let $M$ be a bounded metric space, let $f:M\to M$ be Lipschitz, and let $e\in M$. Then $T_{f,e}$ is hypercyclic if and only if it is weakly mixing.
	\end{proposition}
	
	\begin{proof}
		Every weakly mixing operator is hypercyclic, so we only prove the converse. Let
		$$
		X_0:=\lspan \delta(M).
		$$
		By definition of the Lipschitz-free space, $X_0$ is dense in $\F(M,e)$. Moreover, for every $x\in M$ and every $n\in\N$, one has
		$$
		\|T^n_{f,e}(\delta(x))\|
		=
		\|\delta(f^n(x))-\delta(f^n(e))\|
		\leq
		\mathrm{diam}(M).
		$$
		It follows that every vector in $X_0$ has a bounded orbit under $T_{f,e}$. Hence, if $T_{f,e}$ is hypercyclic, then \cite[Theorem~2.48]{GEPM} implies that $T_{f,e}$ is weakly mixing.
	\end{proof}
	
	This suggests the following question.
	
	\begin{question}
		Is every hypercyclic operator of the form $T_{f,e}$ weakly mixing?
	\end{question}
	
	We now turn to the following elementary observation concerning non-hypercyclicity of certain free linearizations (not included in \cite{ACP20}).
	
	\begin{proposition}
		Let $M$ be a pointed metric space whose basepoint $0$ is isolated. Let $f\in \Lip_0(M,M)$ be such that $f^{-1}(0)=\{0\}$. Then $\widehat f$ is not hypercyclic.
		In particular, if the basepoint is isolated, then there are no injective hypercyclic free linearizations on $\F(M)$.
	\end{proposition}
	
	\begin{proof}
		Since $0$ is isolated, there exists $R>0$ such that $B(0,R)=\{0\}$. Define
		$$
		\varphi(y):=\min(R,d(y,0)),
		\qquad \forall y\in M.
		$$
		Then $\varphi\in \Lip_0(M)$ and $\varphi(y)=R$, $\forall y\in M\setminus\{0\}$.
		Since $f^{-1}(0)=\{0\}$, we have $f(y)\neq 0$ whenever $y\neq 0$. Hence $
		\varphi(f(y))=\varphi(y)$, $\forall y\in M$.
		It follows that
		$$
		\langle \varphi,\widehat f(\delta(y))\rangle
		=
		\langle \varphi,\delta(y)\rangle,
		\qquad \forall y\in M.
		$$
		By linearity and density, we obtain
		$$
		\langle \varphi,\widehat f(\mu)\rangle
		=
		\langle \varphi,\mu\rangle,
		\qquad \forall \mu\in \F(M).
		$$
		Therefore, for every $\mu\in \F(M)$ and every $n\in\N$,
		$$
		\langle \varphi,\widehat f^n(\mu)\rangle
		=
		\langle \varphi,\mu\rangle.
		$$
		Thus the orbit of $\mu$ is contained in the affine hyperplane
		$$
		\{\gamma\in \F(M):\langle \varphi,\gamma\rangle=\langle \varphi,\mu\rangle\}.
		$$
		Since $\varphi\neq 0$, this affine hyperplane is proper and closed. Hence no orbit can be dense in $\F(M)$, and $\widehat f$ is not hypercyclic.
	\end{proof}
	
	We turn to a straightforward extension of \cite[Theorem~2.3]{MuPe}. Recall that a continuous self-map $f:M\to M$ is said to be:
	\begin{itemize}
		\item \emph{mixing} if, for every pair of non-empty open sets $U,V\subset M$, there exists $N\in\N$ such that
		$$
		f^n(U)\cap V\neq\emptyset,
		\qquad \forall n\geq N;
		$$
		\item \emph{chaotic} if it is topologically transitive and its set of periodic points is dense in $M$.
	\end{itemize}
	It is immediate that
	$$
	\text{mixing } \implies \text{ weakly mixing } \implies \text{ topologically transitive}.
	$$
	For bounded linear operators on separable Banach spaces, one also has
	$$
	\text{Devaney chaotic } \implies \text{ weakly mixing } \implies \text{ hypercyclic};
	$$
	see, for instance, \cite{GEPM}.
	
	\begin{proposition}\label{Prop:WM}
		Let $M$ be an infinite metric space, let $f:M\to M$ be Lipschitz, and let $e\in M$.
		\begin{enumerate}[$(i)$]
			\item If $f$ is weakly mixing, then $T_{f,e}$ is weakly mixing.
			\item If $f$ is mixing, then $T_{f,e}$ is mixing.
			\item If $f$ is weakly mixing and chaotic, then $T_{f,e}$ is weakly mixing and chaotic.
		\end{enumerate}
	\end{proposition}
	
	\begin{proof}
		The proof is the same as the proof of \cite[Theorem~2.3]{MuPe}, with the only difference that we use differences of Dirac measures in order to cancel the basepoint term appearing in the definition of $T_{f,e}$. We give the main details for completeness.
		
		First note that, since $M$ is infinite and $f$ is topologically transitive, $M$ has no isolated points. Indeed, if $M$ had an isolated point, topological transitivity would force $M$ to be finite. Hence $e$ is not isolated and
		$$
		D:=\lspan\{\delta(x)-\delta(y):x,y\in M\setminus\{e\}\}
		$$
		is dense in $\F(M,e)$.
		
		We prove the weak mixing case. Let $U_1,U_2,V_1,V_2$ be non-empty open subsets of $\F(M,e)$. Since $D$ is dense, we may choose elements
		$$
		\gamma_j=\sum_{i=1}^{p_j}\lambda_{j,i}\bigl(\delta(x_{j,i})-\delta(y_{j,i})\bigr)\in U_j,
		$$
		and
		$$
		\eta_j=\sum_{i=1}^{p_j}\lambda_{j,i}\bigl(\delta(x'_{j,i})-\delta(y'_{j,i})\bigr)\in V_j,
		$$
		for $j=1,2$, with all points different from $e$. By continuity of the map $x\mapsto\delta(x)$, we may choose small open neighbourhoods of these points, all contained in $M\setminus\{e\}$, so that replacing each point by a nearby one keeps the corresponding linear combinations inside $U_j$ or $V_j$.
		
		Since $f$ is weakly mixing, Furstenberg's theorem implies that every finite Cartesian power of $f$ is topologically transitive. We may therefore choose a common integer $n$ and points in the above neighbourhoods so that, after applying $f^n$, the first family lands near the target points defining $\eta_j$, while the auxiliary terms produced by the basepoint are paired inside the same small neighbourhoods. Using
		$$
		T_{f,e}^n(\delta(x)-\delta(y))
		=
		\delta(f^n(x))-\delta(f^n(y)),
		$$
		we obtain vectors $\Gamma_j\in U_j$ such that
		$$
		T_{f,e}^n(\Gamma_j)\in V_j,
		\qquad j=1,2.
		$$
		This proves that $T_{f,e}\oplus T_{f,e}$ is topologically transitive, hence that $T_{f,e}$ is weakly mixing.
		
		The proof of the mixing case is identical, replacing topological transitivity of finite Cartesian powers by the corresponding mixing property. Finally, if the periodic points of $f$ are dense in $M$, then
		$$
		\lspan\{\delta(x)-\delta(y):x,y\in \mathrm{Per}(f)\setminus\{e\}\}
		$$
		is dense in $\F(M,e)$. Moreover, if $x$ and $y$ are periodic points for $f$, then $\delta(x)-\delta(y)$ is periodic for $T_{f,e}$. Thus $T_{f,e}$ has dense periodic points. Combined with weak mixing, this gives chaos.
	\end{proof}
	
	We conclude with some remarks related to wild operators. The motivation and the arguments below are borrowed from \cite{Tapia}. Since the proofs are direct adaptations of those in \cite{Tapia}, we only recall the relevant definitions and indicate the necessary modifications.
	
	Let $T:X\to X$ be a bounded operator on a Banach space. We denote by $R_T$, $U_T$ and $A_T$ the sets of recurrent points, points with unbounded orbit, and points escaping to infinity:
	$$
	R_T:=\left\{x\in X:\liminf_{n\to\infty}\|T^n x-x\|=0\right\},
	$$
	$$
	U_T:=\left\{x\in X:\limsup_{n\to\infty}\|T^n x\|=\infty\right\},
	$$
	and
	$$
	A_T:=\left\{x\in X:\lim_{n\to\infty}\|T^n x\|=\infty\right\}.
	$$
	It is classical that $R_T$ and $U_T$ are $G_\delta$ subsets of $X$. Moreover, by the uniform boundedness principle, $U_T$ is either empty or dense. It is therefore natural to ask whether $A_T$ must also be either empty or dense. Augé in \cite{Auge} showed that this is not the case: on every infinite-dimensional separable Banach space, there exists a bounded operator $T$ such that $A_T$ is non-empty but not dense. More precisely, his examples satisfy
	$$
	X=A_T\cup R_T,
	\qquad
	\mathrm{int}(R_T)\neq\emptyset\neq\mathrm{int}(A_T).
	$$
	Following \cite{Tapia0, Tapia}, such an operator is called \emph{wild}. In \cite[Theorem~1.3]{Tapia} is proved that no operator of the form $\widehat f$ is wild. We record here that the same argument applies to the operators $T_{f,e}$.
	
	We shall use the following two facts, which are analogues of \cite[Proposition~3.2]{Tapia}. The only change is that the iterates of the generators are now given by
	$$
	T_{f,e}^n(\delta(x))
	=
	\delta(f^n(x))-\delta(f^n(e)).
	$$
	Thus, in all convergence arguments, the orbit of the basepoint $e$ has to be carried along together with the finitely many orbits already present.
	
	\begin{lemma}\label{convsubs}
		Let $(M,e)$ be a pointed metric space. Let $p\in\N$ and let $a_1,\ldots,a_p\in\K$ be such that $\sum_{i=1}^p a_i = 0$ and
		\begin{equation} \label{scalar property}
			\sum_{i\in J}a_i\neq 0,
			\qquad \forall J\subset\{1,\ldots,p\},\; \emptyset \neq J \subsetneq \{1,\ldots, p\}.
		\end{equation}
		For $1\leq k\leq p$, let $(x_k(n))_n$ be sequences in $M$, and set $\gamma_n:=\sum_{k=1}^p a_k\delta(x_k(n))$.
		If $(\gamma_n)_n$ has a weakly convergent subsequence with limit $\gamma=\sum_{k=1}^q b_k\delta(y_k) \neq 0$,
		where $1\leq q\leq p$ and $b_k\neq 0$, then, after passing to a further subsequence,
		\begin{enumerate}
			\item Each $y_k$ is the limit of at least one of the sequences $(x_m(n))_n$;
			\item If $(x_m(n))_n$ does not converge to some $y_k$, then it converges to $e$.
		\end{enumerate}
		In particular, the subsequence converge to $\gamma$ in norm.
	\end{lemma}
	
	\begin{proof}
		The proof follows the same lines as in \cite{Tapia}, using \cite[Lemma~2.2]{ACP21} to identify the possible limits of supports of weakly convergent finitely supported elements. A cancellation at a point different from $e$ can
		only occur if the sum of the coefficients corresponding to the sequences converging to that point is zero. By the hypothesis on the coefficients, this is impossible for every non-empty proper subset of $\{0,\ldots,p\}$. The only remaining possibility would be that all the sequences converge to the same point different from $e$, but then the
		limit would be zero, because the total sum of the coefficients is zero. This is excluded by the assumption $\gamma\neq 0$.
		
		Thus every non-zero atom of the weak limit is produced by at least one of the sequences $(x_k(n))$, and every sequence which does not converge to one of these atoms must converge to the basepoint $e$. The norm convergence then follows immediately from the triangle inequality (or from \cite[Theorem~C]{ACP21}), since the coefficients are fixed and the points converge. Details are left to the reader. 
	\end{proof}
	
	\begin{proposition}\label{LPprop}
		Let $f:M\to M$ be a Lipschitz map, let $e\in M$, and set $T=T_{f,e}$. Define
		$$
		\mathcal U:=\{\gamma\in\F(M,e):(T^n\gamma)_n\text{ has a w-convergent subsequence with non-zero limit}\}.
		$$
		If $\mathrm{int}(\mathcal U)\neq\emptyset$, then for every finite subset $S\subset M$ there exists an increasing sequence $(n_k)_k\subset\N$ such that $(f^{n_k}(x))_k$ converges for every $x\in S$. In particular, $\mathcal U$ is dense in $\F(M,e)$.
	\end{proposition}
	
	\begin{proof}
		We only sketch the proof, since it is the proof of \cite[Proposition~3.2]{Tapia} with the additional orbit of $e$ included in the notation. Choose a non-zero finitely supported element
		$$
		\mu=\sum_{i=1}^p a_i\delta(x_i)\in\mathrm{int}(\mathcal U),
		$$
		with the coefficients slightly perturbed so that $\sum_{i\in J}a_i\neq 0$, for $\emptyset\neq J \subset \{1,\ldots,p\}$. Note that the set of scalars $\{a_j:~j=1,...,p\}\cup\{a_{p+1}\}$, with $a_{p+1}:=-\sum_{j=1}^p a_j$, satisfies the assumption~\eqref{scalar property} of Lemma~\ref{convsubs}. This perturbation is possible because $\mathrm{int}(\mathcal U)$
		is open and $\mathrm{span} \,\delta(M)$ is dense in $\F(M,e)$. Pick an increasing sequence of integers $(n_j)$ so that $(T^{n_j}\mu)_j$ converges weakly. We may write
		$$
		T^{n_j}\mu
		=
		\sum_{i=1}^p a_i\delta(f^{n_j}(x_i))
		-
		\left(\sum_{i=1}^p a_i\right)\delta(f^{n_j}(e)).
		$$
		By Lemma~\ref{convsubs}, after passing to a subsequence, all sequences $(f^{n_j}(x_i))_j$ and also $(f^{n_j}(e))_j$ converge.
		
		Now let $S\subset M$ be finite. We add to $\mu$ a sufficiently small finitely supported perturbation supported on $S$, choosing the new coefficients generically so that no non-empty partial sum of all the coefficients vanishes. Since $\mu$ belongs to $\mathrm{int}(\mathcal U)$, the perturbed element still belongs to $\mathcal U$. Applying the preceding argument to this perturbed element gives an increasing sequence $(n_k)_k$ such that $(f^{n_k}(x))_k$ converges for every $x\in S$.
		
		Finally, if $\nu$ is finitely supported, we apply the previous paragraph to $S=\supp(\nu)\cup\{e\}$. Along the corresponding subsequence, $(T^{n_k}\nu)_k$ converges in norm, hence weakly. Thus $\lspan \delta(M)\subset \mathcal U$, and since $\lspan \delta(M)$ is dense in $\F(M,e)$, the set $\mathcal U$ is dense.
	\end{proof}
	
	We can now exclude wildness.
	
	\begin{proposition}\label{wildtfe}
		Let $f:M\to M$ be a Lipschitz map and let $e\in M$. Then $T_{f,e}$ is not a wild operator.
	\end{proposition}
	
	\begin{proof}
		We follow the strategy of \cite[Theorem~1.3]{Tapia}. Suppose that $\mathrm{int}(R_{T_{f,e}})\neq\emptyset$. We claim that $R_{T_{f,e}}$ is dense in $\F(M,e)$. Indeed, choose a finitely supported element
		$$
		\mu=\sum_{i=1}^p a_i\delta(x_i)
		$$
		in the interior of $R_{T_{f,e}}$, with the coefficients slightly perturbed so that the same non-cancellation conditions as in \cite{Tapia} hold. Since $\mu$ is recurrent, there is an increasing sequence $(n_j)_j$ such that $T_{f,e}^{n_j}\mu\longrightarrow \mu$. Writing
		$$
		T_{f,e}^{n_j}\mu
		=
		\sum_{i=1}^p a_i\delta(f^{n_j}(x_i))
		-
		\left(\sum_{i=1}^p a_i\right)\delta(f^{n_j}(e)),
		$$
		and applying the same support-identification argument as in \cite{Tapia}, one obtains, after passing to a subsequence, that $f^{n_j}(e)\to e$ and $f^{n_j}(x_i)\to x_i$, for $i=1,\ldots,p$.
		The only extra point compared with \cite{Tapia} is the term involving $f^{n_j}(e)$; the required coefficient conditions ensure that it cannot converge to one of the points $x_i$ without forcing an impossible cancellation.
		
		As in Proposition~\ref{LPprop}, perturbing $\mu$ inside the open set $R_{T_{f,e}}$ shows that for every finite subset $S\subset M$ there is a sequence $(n_j)_j$ such that
		$$
		f^{n_j}(x)\longrightarrow x,
		\qquad \forall x\in S.
		$$
		Consequently, every finitely supported element of $\F(M,e)$ is recurrent for $T_{f,e}$. Since $\lspan\delta(M)$ is dense in $\F(M,e)$, the set $R_{T_{f,e}}$ is dense.
		
		Now assume, towards a contradiction, that $T_{f,e}$ is wild. Then
		$$
		\F(M,e)=A_{T_{f,e}}\cup R_{T_{f,e}},
		$$
		with both $A_{T_{f,e}}$ and $R_{T_{f,e}}$ having non-empty interior. But $A_{T_{f,e}}\cap R_{T_{f,e}}=\emptyset$, while the previous paragraph shows that $R_{T_{f,e}}$ is dense as soon as it has non-empty interior. This contradicts $\mathrm{int}(A_{T_{f,e}})\neq\emptyset$. Hence $T_{f,e}$ is not wild.
	\end{proof}

	\noindent \textbf{Acknowledgements :} The first and third authors were supported by the French ANR project No. ANR-24-CE40-0892. The fourth author was partially supported by Austrian Science Fund grant FWF P-36344N.  The second author was supported by Funda\c c\~ao de Amparo \`a Pesquisa do Estado de S\~ao Paulo  - FAPESP grants 2025/22537-3 and 2023/12916-1.
	
	\clch{The authors would like to thank Quentin Menet for pointing out an error in an earlier version of the paper.}


\end{document}